\documentclass[a4paper,duplex]{article}

\usepackage{amsmath}
\usepackage{amssymb}
\usepackage{amsthm}
\usepackage{color}
\usepackage{epsfig}
\usepackage{algpseudocode}
\usepackage{cite}
\usepackage{txfonts}
\usepackage{enumerate}
\usepackage{supertabular}
\usepackage{epsf}

\newtheorem{defn}{Definition}[section]
\newtheorem{lemma}[defn]{Lemma}
\newtheorem{proposition}[defn]{Proposition}

\newtheorem*{theorema}{Theorem 1}
\newtheorem*{theoremb}{Theorem 2}

\newtheorem{corollary}[defn]{Corollary}

\theoremstyle{definition}

\newcommand{\slex}{\pi}
\newcommand{\evaluate}[1]{\tau(#1)}

\newcommand{\path}[2]{[#1,#2]}

\newcommand{\vertex}[1]{\mathsf{#1}}

\newcommand{\canc}[2]{\Delta(#1,#2)}

\newcommand{\floor}[1]{\left\lfloor #1 \right\rfloor}

\newcommand{\ceil}[1]{\left\lceil #1 \right\rceil}

\newcommand{\superscript}[1]{\ensuremath{^{\textrm{#1}}}}

\renewcommand{\th}[0]{\superscript{th}}

\begin{document}

\title{The conjugacy problem in hyperbolic groups for finite lists of group
elements}

\author{D.J.\ Buckley \& D.F.\ Holt}

\maketitle

\begin{abstract}
Let $G$ be a word-hyperbolic group with given finite generating set,
for which various standard structures and constants have been pre-computed.
A (non-practical) algorithm is described that,
given as input two lists $A$ and $B$, each composed of $m$ words in the
generators and their inverses, determines whether or not the lists are
conjugate in $G$, and returns a conjugating element should one exist.
The algorithm runs in time $O(m\mu)$, where $\mu$ is an upper bound on
the lengths of elements in the two lists. Similarly, an algorithm is outlined
that computes generators of the centraliser of $A$,
with the same bound on running time.
\end{abstract}

\section{Introduction}

In \cite{bridson2005conjugacy}, Bridson and Howie give a solution of the conjugacy problem for finite lists $A = (a_1, \ldots, a_m)$ and $B = (b_1, \ldots, b_m)$ of elements in a word-hyperbolic group -- in fact, they prove that the problem is solvable in time $O(m \mu^2)$ for any fixed torsion-free word-hyperbolic group, where $\mu$ is an upper bound on the length of elements in the two lists.

The aim here is both to improve the bound on running time to $O(m\mu)$, and to tie up the rather limp conclusion in part 2 of Theorem B of \cite{bridson2005conjugacy}, in which their algorithm terminates
without giving any results on the conjugacy
when the lists consist entirely of elements of finite order.
The general algorithm for the conjugacy problem for finite lists described in
\cite{bridson2005conjugacy} is almost certainly at least exponential in the input length.

The ideas used here closely relate to those in \cite{epstein2006linearity}, in which Epstein and Holt show that the conjugacy problem for single elements in a word-hyperbolic group can be solved in linear time if one assumes a RAM model of computing. They do so by showing that infinite order elements tend to be well-behaved when raised to large powers, and finite order elements can be conjugated to elements of short length whose conjugacy can be precomputed.
In fact we will adapt and make use of a number of results from that paper.

The results in this paper are covered in more detail in
\cite{buckley2010thesis}. 
Our main theorem is:

\begin{theorema}
\label{thm:linear_conj}
Given a word-hyperbolic group $G= \langle X \mid R \rangle$,
there is an algorithm that, given a number $m \ge 1$ and lists
$A = (a_1, \ldots, a_m)$ and $B = (b_1, \ldots, b_m)$,
each containing words in $X \cup X^{-1}$, either finds an element $g \in G$
such that $A^g =_G B$ or determines that no such element exists.
The algorithm runs in time $O(m\mu)$, where $\mu$ is an upper bound on the
lengths of elements in the lists.
\end{theorema}

Due to the exhaustive search required to verify that two lists are not
conjugate, the method will in fact enable the computation of all conjugating
elements -- in particular, a simple modification yields the following
additional result.

\begin{theoremb}
\label{thm:linear_cent}
Given a word-hyperbolic group $G= \langle X \mid R \rangle$,
there is an algorithm that, given a number $m \ge 1$ and a list $A = (a_1, \ldots, a_m)$ containing words in $X \cup X^{-1}$, returns a generating set for
the centraliser $C_G(A)$.
The algorithm runs in time $O(m\mu)$, where $\mu$ is an upper bound on the
lengths of elements in the list.
\end{theoremb}

As in~\cite{epstein2006linearity}, our complexity estimates are based on a RAM
model of computing, in which the basic arithmetical operations on integers
are assumed to take constant time. An alternative model with the same
complexity involves Turing machines that have multiple tapes, and may have
multiple heads on each tape (both the number of tapes and the number of heads will be in $O(1)$). The heads are independent; that is, while they all start in the same place, they need not behave in the same way: one may be moved and used to read and write on the tape while another remains stationary and later moves to read said area of tape. This model is described in\cite{galil1976real}.



Throughout this paper, we assume that $G$ is a fixed word-hyperbolic group with
fixed generating set $X$, where we assume for convenience that $X=X^{-1}$.
The pre-computations that we need to carry out in $G$ will be summarised in Section~\ref{section:notation}. All of the constants
referred to explicitly or implicitly will depend on $G$ and $X$ only. 

The technicalities behind the proof in the case where one element, say $a_1$, has infinite order are largely covered by solving the conjugacy problem $a_1^h =_G b_1$ for $h$ as in \cite{epstein2006linearity}. In the process of doing so, a useful description of elements of the centraliser $C$ of $a_1$ is found, and then used to test if $A^{ch} =_G B$ for some $c \in C$. Of course $C$ is infinite, so it is important to perform this test efficiently. Section \ref{section:inf_order} describes a way of doing so.

These methods cannot be used when both lists consist entirely of torsion
elements.
It is, however, possible to show that, if $A$ and $B$ have length $m$,
then a pair of lists $A'$ and $B'$ can be efficiently found such that $A^h =_G B$ if and only if $A'^h =_G B'$, and such that either $A'$ or $B'$ contains an
infinite order element, or each element in $A'$ and $B'$ has length at most
a number $L(m)$.  

The number $L(m)$ grows exponentially with $m$.
However, it can be shown that there is a constant $n$ such that, if the lists
consist of distinct torsion elements and have length at least $n$,
then their centralisers are finite and of bounded order. In particular, there are only a bounded number of elements that can simultaneously conjugate the
first $n$ elements of $A$ to the the first $n$ elements of $B$, and so testing each of these conjugating elements on the remainder of the elements in $A$ and $B$ completes the procedure.
Since $L(n)$ is a constant, we can use the general algorithm given in \cite{bridson2005conjugacy} to find these conjugating elements in constant time.

\section{Notation}

\label{section:notation}

We shall occasionally use the notation $x =^d y$ to mean $|x-y| \le d$.

A very brief introduction to hyperbolicity and some definitions included for
convenience are sketched below. The reader is referred to
\cite{alonso1990notes} for a more detailed introductory treatment of
the theory of (word-)hyperbolic groups.

A \textit{path} $\alpha : \path{a}{b} \to S$ is an arc-length parametrization of a connected curve in a metric space $S$. If $\alpha$ is described as connecting a point $x$ to a point $y$ then $\alpha(a) = x$ and $\alpha(b) = y$; it will normally be assumed that $a=0$ in this case. If $x = \alpha(t)$ for some $t \in \path{a}{b}$ then write $x \in \alpha$. If $x = \alpha(c)$ and $y = \alpha(d)$ for $c, d \in \path{a}{b}$, write $\path{x}{y}$ to denote the restriction $\alpha|_{\path{c}{d}}$ and then define $d_\alpha(x, y) = d - c$. This definition is a little loose where $\alpha$ is not a simple curve; in order to deal with this ambiguity assume that whenever a point $x \in \alpha$ is picked, a specific value $t_x \in \path{a}{b}$ with $\alpha(t_x) = x$ is also picked for use with these definitions.

A path $\alpha$ is \textit{$(\lambda, \epsilon)$-quasigeodesic} (with
$\lambda \ge 1$, $\epsilon \ge 0$) if $d_\alpha(x, y) \le \lambda d(x, y) + \epsilon$ for all $x, y \in \alpha$.
For $L>0$, the path is \textit{$L$-local $(\lambda, \epsilon)$-quasigeodesic},
if all subpaths of length at most $L$ are $(\lambda, \epsilon)$-quasigeodesic.
It is \textit{geodesic} if it is $(1, 0)$-quasigeodesic. A \textit{geodesic metric space} is a metric space in which each pair of points is connected by a geodesic.
A \textit{geodesic triangle} in a metric space is a collection of three points (the corners) along with three geodesic paths, one path connecting each pair of corners.

The Gromov inner product of points $x$ and $y$ at a point $z$ in a metric space is defined as \[ (x, y)_z := \frac{d(x, z) + d(y, z) - d(x, y)}{2}. \]

Suppose that $x, y, z$ are points in a geodesic metric space $\Gamma$ and that $\alpha$ and $\beta$ are sides of a geodesic triangle connecting these three points, chosen so that $\alpha(0) = \beta(0) = z$. If $0 \le t \le (x, y)_z$ then the points $\alpha(t)$ and $\beta(t)$ are said to \textit{correspond}.
By making the corresponding definition at the
remaining two corners, each point on the sides of the triangle has a corresponding point on at least one other side (though in degenerate cases, for example when $t=0$, a point may correspond to itself). The triangle is \textit{$\delta$-thin} if $d(r, s) \le \delta$ whenever $r$ and $s$ are corresponding points. A geodesic metric space $\Gamma$ is \textit{$\delta$-hyperbolic} if all geodesic triangles in $\Gamma$ are $\delta$-thin.


Given a group $G$ with generating set $X$, the Cayley graph $\Gamma$ of $G$ is the graph with vertex set $G$ and edges connecting $g$ to $gx$ whenever $g \in G$ and $x \in X$, endowed with the metric that sets each edge to have length 1 (often called the ``word metric''). A \textit{word-hyperbolic group} is a finitely generated group in which all geodesic triangles in its Cayley graph are $\delta$-thin for some fixed $\delta \ge 0$.
It turns out that the property of being word-hyperbolic is independent of generating set, though the value of $\delta$ is not; see \cite{alonso1990notes}.

Throughout this paper, we assume that an ambient finitely generated group $G$ has been fixed along with a finite inverse-closed generating set $X$, and that $G$ is $\delta$-hyperbolic for some $\delta$ with respect to this generating set.
For our later convenience, we assume that $\delta \ge 1$.
Where a value is said to be ``bounded'' or ``in $O(1)$'', the value is bounded above by some constant that depends only on $G$ and $X$.

All geometric constructions occur inside the Cayley graph $\Gamma$ of $G$ with respect to $X$, inside which the vertex $\vertex{1}$ represents the identity element of $G$.

A \textit{word} is a finite sequence of elements of $X$, written as a concatenation.
The {\em length} $|w|$ of a word $w$ is the length of the sequence of generators
that defined $w$.
For each $1 \le a \le |w|$, denote the $a$\th{} letter of $w$ by $w[a]$. For each $0 \le a \le |w|$, write $w(a) = w[1] w[2] \cdots w[a]$ to refer to the subword given by the first $a$ letters of $w$ and let $w(a:b) = w[a+1] w[a+2] \cdots w[b]$ so that $w(b) = w(a) w(a:b)$ whenever $0 \le a \le b \le |w|$.

One operation that we shall use frequently is the {\em half-cyclic conjugate} of a word.
Given a word $w = a_1 \cdots a_n$, let $l := \floor{\frac{n}{2}}$, let $w_L := w(l)$ and $w_R := w(l:n)$.
Then the half-cyclic conjugate is defined as $w_C := w_R w_L$. For example, if $w = abcde$ then $w_C = cdeab$.

Given a starting vertex in $\Gamma$, a word $w$ uniquely labels a path in $\Gamma$. By taking $\vertex{1}$ as the starting vertex, each word defines an element $\evaluate{w}$ of the group.
If two words $u$ and $v$ map to the same element of $G$, write $u =_G v$. The length of an element $g \in G$, written $|g|_G$, is the minimum length of a word that defines $g$ and, for a word $w$, we define $|w|_G := |\tau(w)|_G$.
A word is \textit{geodesic} if $|w| = |w|_G$, that is, $w$ is a shortest representative of $\tau(w)$.

The generating set $X$ is assumed to be ordered, so that the notion of the shortlex least representative word $\slex(w)$ for each group element $g=\tau(w)$
exists (that is, the lexicographically least word among
all geodesic words that define $g$).
A word $w$ is said to be \textit{shortlex reduced} if $\slex(w)=w$.
A \textit{straight} word $w$ is one for which $|w^n|_G = |w^n|$ for any positive integer $n$.
Similarly, a \textit{shortlex straight} word is one for which $w^n$ is shortlex reduced for any positive integer $n$.

In \cite{epstein2006linearity}, the following result due to Shapiro is proved:

\begin{lemma}
\label{lemma:short_lex}
There is an algorithm that, given a word $w$, returns $\slex(w)$ in time
$O(|w|)$.
\end{lemma}

This algorithm enables a number of other operations to be computed in linear
time; for example, testing  equality of words in $G$, and whether a given word
represents the identity.

In order to use the results from \cite{epstein2006linearity}, it is assumed
that various constructions related to the group (such as the shortlex word
acceptor) have been pre-computed.
The constants defined below, which are bounded in terms of $\delta$ and $|X|$,
will be used throughout the paper.

\begin{itemize}
\item $L := 34\delta + 2$
\item $V$, the number of vertices in the closed $2\delta$-ball around $\vertex{1}$ (so $|V| \le |X|^{1+2\delta}$).
\item $M := 20\delta^2V^3L^2$
\end{itemize}

\section{The infinite order case}

\label{section:inf_order}

We shall say that a word $w$ has infinite order if the element $\tau(w)$ in
$G$ that it represents has infinite order.
Recall that we are given two lists of words $A = (a_1, \ldots, a_m)$ and
$B = (b_1, \ldots, b_m)$ that we wish to test for conjugacy in $G$.
The aim of the section is to prove Theorems 1 and 2 under the additional
assumption that $a_1$ has infinite order.

The method is a combination of those described in \cite{epstein2006linearity} and \cite{bridson2005conjugacy}. The following three subsections concern testing conjugacy between single elements only; Section \ref{section:holt} is just a summary of some of \cite{epstein2006linearity}. The motivation here is to apply these methods to $a_1$ and $b_1$, since any element conjugating $A$ to $B$ must necessarily conjugate $a_1$ to $b_1$.

\subsection{Results from \cite{epstein2006linearity}}

\label{section:holt}

It is proved in \cite[Section 3]{epstein2006linearity}
that the conjugacy problem for single elements is solvable in time linear in the total input length. The proof has several steps. The first few will be followed here as well; they are outlined in this subsection.

The authors of \cite{epstein2006linearity} first show that elements that are ``difficult to shorten'' are actually of infinite order, and behave nicely when raised to large powers.

\begin{proposition} \cite[Lemma 3.1]{epstein2006linearity}
\label{prop:short_conj}
Let $w$ be a shortlex reduced word and let $u = \slex(w_C)$. If $|u|>2L$,
then all positive powers of $u$ label $L$-local $(1, 2\delta)$-quasigeodesics.
\end{proposition}

\begin{proposition}\cite[Proposition 2.3]{epstein2006linearity}
\label{prop:nice_qgeo}
If $w$ is an $L$-local $(1, 2\delta)$-quasigeodesic path in $\Gamma$, and $u$ is
a geodesic path connecting its endpoints, then every point on $w$ is within $4\delta$ of a point on $u$, and every point on $u$ is within $4\delta$ of a point on $w$.
Also, if $|w| \ge L$ then $|u| \ge \frac{7|w|}{17}$.
\end{proposition}

In particular, if $|w_C|_G > 2L$ then $w$ has infinite order,
since there is no bound on the length of shortest representatives of its powers.

The next step is to show that, for such a word $w$, a conjugate of a power of
$w$ that is equal in $G$ to a shortlex straight element can be efficiently found.
The following two results summarise Section 3.2 of \cite{epstein2006linearity}.

\begin{proposition}
\label{prop:long_conj_slex}
Suppose $u$ is a shortlex reduced word with $|u| > L$, such that all positive powers of $u$ label $L$-local $(1, 2\delta)$-quasigeodesics. Then there exists an integer $0 < k \le V^4$ and a word $a$ with $|a| \le 4\delta$, such that $\slex(a^{-1}u^ka)$ is shortlex straight.
\end{proposition}

In \cite{epstein2006linearity}, $k$ is shown to be less than $Q^2$ where $Q$ is the number of group elements in the $4\delta$-ball around $\vertex{1}$, but $Q \le V^2$, so our statement is slightly weaker.

\begin{proposition}
\label{prop:test_sls}
Given a shortlex reduced word $u$, testing if $u$ is shortlex straight takes time $O(|u|)$.
\end{proposition}

Finding the shortlex straight conjugate of a power is thus just a case of exhaustively testing each $k$ and $a$ as in Proposition \ref{prop:long_conj_slex}. Once a word is shortlex straight, it is easier to test conjugacy against it. The next result summarises Section 3.3 of \cite{epstein2006linearity}.

\begin{proposition}
\label{prop:test_conj_sls}
If $u$ is shortlex straight, $v$ is a word with $|v|_G > L$,
such that all positive powers of $v$ are
$(1, 2\delta)$ $L$-local quasigeodesics,
and $g^{-1}vg =_G u$ for some $g$, then there exists a word $h$ with
$|h| \le 6\delta$ such that $\slex(h^{-1}vh)$ is a cyclic conjugate of $u$.
\end{proposition}

In \cite{epstein2006linearity}, the authors test whether a word $u$ is a cyclic conjugate of another word $v$ by testing if $v$ appears as a substring of $u^2$, using the Knuth-Morris-Pratt algorithm.
The standard implementation of this algorithm involves a lookup table of size $O(|u|)$, so might be imagined to take time $O((|u| + |v|)\log(|u|))$ on a Turing machine.
An alternative implementation on a multi-head Turing machine that runs in
time $O(|u| + |v|)$ is presented in~\cite{galil1976real}.

A refinement of the proof of Proposition \ref{prop:test_conj_sls} gives a nice form for elements of the centraliser of a shortlex straight word. This result summarises Section 3.4 of \cite{epstein2006linearity}.

\begin{proposition}
\label{prop:find_centraliser_sls}
If $z$ is shortlex straight and $y^l=z$ with $l \ge 1$ maximal, then $g\in C_G(z)$ implies that $g=_Gy^iy_1h$, with $y_1$ a prefix of $y$, $i\in \mathbb{Z}$ and $|h| \le 2\delta$. The prefix $y_1$ depends only on $h$.
Furthermore, $l$, $y$ and the set of words $y_1h$ can be computed in time $O(|z|)$.
\end{proposition}

That completes the information that will be required from \cite{epstein2006linearity}; the next proposition summarises this section.

\begin{proposition}
\label{prop:eh_solve_conj}
There exists an algorithm which, given shortlex reduced words $u$ and $v$ with $|u_C|_G > 2L$ and $|v_C|_G > 2L$,
computes words $a$ and $y$, and a set $S$ of at most $V$ words,
such that $y$ is shortlex straight, and $u^g =_G v$ implies that $g =_G ay^ns$ for some $s \in S$. All output words have length in $O(|u| + |v|)$ and the algorithm runs in time $O(|u| + |v|)$.
\end{proposition}

\begin{proof}
Proposition \ref{prop:short_conj} implies that all positive powers of both $\slex(u_C)$ and $\slex(v_C)$ label $L$-local $(1, 2\delta)$-quasigeodesics. Applying Proposition \ref{prop:long_conj_slex} implies that there is a word $a'$ of length at most $4\delta$ and a positive integer $i \le V^4$ with $z := \slex(((u_C)^i)^{a'})$ shortlex straight.
Since both $|a'|$ and $i$ are in $O(1)$ and testing if $z := \slex(((u_C)^i)^{a'})$ is shortlex straight takes time $O(|u|)$, a specific $a'$ and $i$ can be found in time $O(|u|)$.

Using the K-M-P algorithm from \cite{galil1976real}, we find the second instance of $z$ as a substring of $z^2$. If this match is found at position $j$ then $z = z(j:|z|) z(j)$, so $z = (z(j))^l$ for some $l$ and $l$ is maximal for words of this form. Let $y = z(j)$; then $y$ is also shortlex straight.

If $u$ is conjugate to $v$ then $u^i$ is conjugate to $v^i$ and so $z$ is conjugate to $(v_C)^i$. Applying Proposition \ref{prop:test_conj_sls} implies that,
if this is the case, then $((v_C)^i)^b$ is equal in $G$ to a cyclic conjugate of $z$ for some word $b$ with $|b| \le 6\delta$. Test for all words $b$ with $|b| \le 6\delta$ whether $\slex(((v_C)^i)^b)$ is a substring of $z^2$ using the K-M-P algorithm again. There are $O(1)$ tests, each taking time $O(|u| + |v|)$,
so a specific $b$ satisfying this property, if one exists, can be found in time $O(|u| + |v|)$. If all tests fail, $u$ and $v$ are not conjugate so the algorithm stops. Otherwise a subword $z(k)$ is found such that $((v_C)^i)^{bz(k)^{-1}} =_G z$. Let $c = z(k)b^{-1}$ for the first $b$ found and continue.

Apply Proposition \ref{prop:find_centraliser_sls} to compute a set $S'$ of words $y_1h$ such that $z^d =_G z$ implies that $d =_G y^n s'$ for some $n \in \mathbb{Z}$ and $s' \in S'$. This again takes time $O(|u| + |v|)$.

Now suppose that $u^g =_G v$. Note that
$$ z^c =_G  (v_C)^i =_G (v^i)^{v_L} =_G (u^i)^{gv_L}
 =_G  ((u_C)^i)^{(u_L)^{-1}gv_L}
 =_G z^{a'^{-1}(u_L)^{-1}gv_L},$$
so that $a'^{-1}(u_L)^{-1}gv_Lc^{-1} \in C_G(z)$, and so is equal in $G$ to $y^ny_1h$ with $n \ge 0$, and $y_1h \in S'$. Therefore $g=_Gu_La'y^ny_1hcv_L^{-1}$.

Let $a := u_La'$ and $S := \{ y_1hcv_L^{-1} : y_1h \in S' \}$ and the proposition is proved.
\end{proof}

%
\subsection{Finding long powers of infinite order elements}

The aim of this section is to show that, given a word $w$ of infinite order,
there exists an efficiently computable shortlex
reduced word $w'$, which is equal in $G$ to a conjugate of a power of $w$, and for which $|\slex(w'_C)|>2L$.
Given two infinite order words $u$ and $v$, finding these conjugates of powers of $u$ and $v$ allows Proposition \ref{prop:eh_solve_conj} to be applied, thus providing a description of conjugating elements for any pair of infinite order words.

The next three results are reasonably well-known properties of word-hyperbolic groups and hyperbolic spaces;
they are taken from \cite{alonso1990notes} although similar results appear in many other expositions of the subject area. The  values of the constants in our statements are derived from the proofs in \cite{alonso1990notes}.

\begin{proposition}
\label{prop:inf_order_qgeo} \cite[Proposition 3.2]{alonso1990notes}
For any geodesic word $w$ of infinite order, all positive powers of $w$ label $(\lambda, \epsilon)$-quasigeodesics in $\Gamma$, where $\lambda = |w|V$ and $\epsilon = 2|w|^2V^2 + 2|w|V$.
\end{proposition}

\begin{proposition} \cite[Theorem 2.19]{alonso1990notes}
\label{prop:hyperbolic_divergence}
The function $e:\mathbb{R}_{\ge0} \to \mathbb{R}_{\ge0}$ with $e(0) = \delta$ and $e(l) = 2^{\frac{l}{\delta} - 2}$ for $l > 0$ is a divergence function for any $\delta$-hyperbolic space (i.e. given geodesics $\gamma = \path{x}{y}$ and $\gamma' = \path{x}{z}$, if $r, R \in \mathbb{N}$ with $r + R < \min\{|\gamma|, |\gamma'|\}$ and $d(\gamma(R), \gamma'(R)) > e(0)$, and if $\alpha$ is a path from $\gamma(R + r)$ to $\gamma'(R + r)$ lying outside the open ball of radius $R + r$ around $x$, then $|\alpha| > e(r)$).
\end{proposition}

\begin{proposition}\cite[Proposition 3.3]{alonso1990notes}
\label{prop:qgeo_close_to_geo}
In a hyperbolic space with divergence function $e$, given constants $\lambda \ge 1$ and $\epsilon \ge 0$, there exists $D = D(\lambda, \epsilon, e)>0$
such that if $\alpha$ is a $(\lambda,\epsilon)$-quasigeodesic and $\gamma$ is a geodesic starting and ending at the same points as $\alpha$ then every point on $\gamma$ is within a distance $D$ of a point on $\alpha$. It suffices to take $D$ satisfying $e(\frac{D - e(0)}{2}) \ge 4D + 6\lambda D + \epsilon$.
\end{proposition}

These results can be used to find a power $n$ of an infinite order word $w$
such that $|(w^n)_C|_G$ is large.

\begin{proposition}
\label{prop:make_long}
Let $w$ be a geodesic word of infinite order with $|w| \le 2L$.
Then $|(\slex(w^M))_C|_G > 2L$.
\end{proposition}

\begin{proof}
By Proposition \ref{prop:hyperbolic_divergence}, the function $e(0)=\delta$, $e(l)=2^{\frac{l}{\delta}-2}$ for $l > 0$ is a divergence function for $\Gamma$. Proposition \ref{prop:inf_order_qgeo} implies that $w^{2M}$ labels a $(\lambda, \epsilon)$-quasigeodesic $\alpha$ starting at the identity, where $\lambda = |w|V$ and $\epsilon = 2|w|^2V^2 + 2|w|V$.

We show now that $D := 1000\delta^2LV$ is sufficient to solve the equation in
Proposition \ref{prop:qgeo_close_to_geo} with these parameters.
Since $\exp(x) > x^3/6$ for all $x>0$ and $3 \log 2 > 2$, we find that
\[e(\frac{D-\delta}{2}) =
\frac{\exp(500\delta L V \log 2)}{4\sqrt{2}} > 
\frac{1000^3 \delta^3 L^3 V^3}{648 \sqrt{2}} > 10^6\delta^2L^2V^2.\]
Since $|w| \le 2L$, we have
$$4D + 6 \lambda D + \epsilon = 4D + 6 |w| VD + 2|w|^2V^2 + 2|w|V \le
  4D +12LVD + 8L^2V^2 + 4LV.$$
By considering a shortlex reduced word of length at least $2\delta+1$ defining a
path starting at the origin, we see that $V \ge 4\delta+1 \ge 5$, and $L \ge 36$,
so $LV\ge 180$. But we also have $LV \le D/1000$, so
$$4D +12LVD + 8L^2V^2 + 4LV \le 13LVD = 13000\delta^2L^2V^2,$$
and hence $e(\frac{D-\delta}{2}) > 4D + 6 \lambda D + \epsilon$, as claimed.

Recall that $M =20\delta^2 V^3L^2 = V^2LD/50$.
Let $u := \slex(w^M)$ and let $\gamma$ be a geodesic path starting at the
identity vertex $\vertex{x}$ and ending at the vertex
$\vertex{y} := \evaluate{u}$. Let $\alpha$ be the path between these vertices
whose label is $w^M$. By Proposition \ref{prop:qgeo_close_to_geo}, the vertex
$\vertex{p} := \evaluate{u_L}$ on $\gamma$ lies within $D$ of some vertex
$\vertex{p'}$ on $\alpha$.

Now let $a$ be the label of the path along $\alpha$ between $\vertex{x}$ and
$\vertex{p'}$. Let $\vertex{q}$ be the vertex representing $uu_L$ and
$\vertex{q'}$ the vertex representing $ua$. See Figure \ref{figure:make_long}.

\begin{figure}
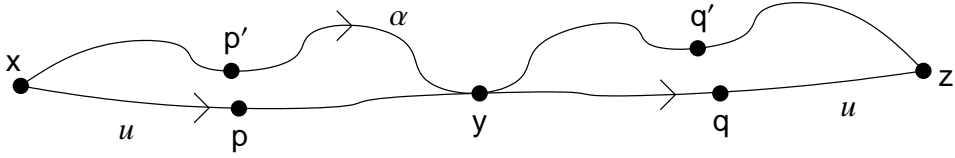

\begin{center}
\input make_long.pstex_t
\caption{Cutting across a long quasigeodesic}
\label{figure:make_long}
\end{center}
\end{figure}

Observe that
$$
        |u_C|
 = d(\vertex{p}, \vertex{q})
\ge d(\vertex{p'},\vertex{q'}) - 2D
\ge \frac{d_\gamma(\vertex{p'}, \vertex{q'})}{\lambda} - \epsilon - 2D 
 =  \frac{|w|M}{\lambda} - \epsilon - 2D.$$

Substituting the values of $M$, $D$, $\lambda$ and $\epsilon$, and using
$|w| \le 2L$, $V \ge 5$, $LV \ge 180$, we have
\begin{eqnarray*}
|u_C| & \ge & LVD/50 - 2|w|^2V^2 - 2|w|V -2D
\ge LV(20\delta^2 LV -8LV -4 -2000\delta^2) \\
& \ge & LV(12\delta^2 LV  -4 -2000\delta^2) > 2L.
\end{eqnarray*}
\end{proof}

The value of $M$ used above is of course by no means optimal (it is probably suboptimal by orders of magnitude) but serves to illustrate that such an explicit bound can be found.

By Proposition \ref{prop:make_long}, short infinite order words can be raised to large powers to obtain words upon which Proposition \ref{prop:eh_solve_conj} may be used. It is useful to confirm that words that are already appropriate inputs stay appropriate when raised to the power of $M$.

\begin{proposition}
\label{prop:long_stays_long}
Suppose that $w$ is a geodesic word, and $|w_C|_G > 2L$. If $n \ge L$ then $|(\slex((w_C)^n))_C| > 2L$. In particular, $|(\slex((w_C)^M))_C| > 2L$.
\end{proposition}

\begin{proof}
Let $u := \slex((w_C)^n)$, and let $\alpha$ be the path starting at
$\vertex{x} := \vertex{1}$ labelled by $\slex(w_C)^{2n}$.
Let $\vertex{y} := \evaluate{u}$ and $\vertex{z} := \evaluate{u^2}$. Now let $\vertex{p} := \evaluate{u_L}$ and let $\vertex{q} := \evaluate{uu_L}$ so that $\vertex{p}$ and $\vertex{q}$ are mid-vertices on the shortlex geodesic paths $\path{\vertex{x}}{\vertex{y}}$ and $\path{\vertex{y}}{\vertex{z}}$ respectively and $u_C$ labels a path from $\vertex{p}$ to $\vertex{q}$. Figure \ref{figure:make_long} provides a suitable diagram once again.

Note that $\alpha$ is an $L$-local $(1, 2\delta)$-quasigeodesic by Proposition \ref{prop:short_conj}, so Proposition \ref{prop:nice_qgeo} applies. Then there is a vertex $\vertex{p'} = \vertex{x} \cdot (w_C)^n(i)$ for some $i$, with $d(\vertex{p'}, \vertex{p}) \le 4\delta$. Let $\vertex{q'} := \vertex{y} \cdot (w_C)^n(i)$ so that $d(\vertex{q'}, \vertex{q}) \le 4\delta$ also. Since $d_\alpha(\vertex{p'}, \vertex{q'}) = n|w_C|_G \ge L$, Proposition \ref{prop:nice_qgeo} also gives a lower bound on $d(\vertex{p'}, \vertex{q'})$ as follows:
\[
        d(\vertex{p}, \vertex{q})
 =^{8\delta}  d(\vertex{p'},\vertex{q'})
\ge \frac{7}{17}d_\alpha(\vertex{p'}, \vertex{q'})
 =  \frac{7}{17}n|w_C|_G 
 >  \frac{14}{17}Ln.
\]

But then
\[
        |(\slex((w_C)^n))_C|
  =   |u_C|
  =  d(\vertex{p}, \vertex{q})
> \frac{14}{17}Ln - 8\delta
\ge \frac{14}{17}L \times 34\delta - 8\delta
\ge 2L
\] as required.
\end{proof}

By the above two results $|(\slex((u_C)^M))_C|_G > 2L$ for any infinite order
geodesic word $u$. Combining this fact with Proposition \ref{prop:eh_solve_conj},
we get:

\begin{proposition}
\label{prop:conj_candidates}
There exists an algorithm which, given geodesic infinite order words
$u$ and $v$, computes words $a$ and $y$, and a set $S$ of at most $V$ words,
such that $y$ is shortlex straight and $u^g =_G v$ implies that $g =_G ay^ns$ for some $s \in S$ and $n \in \mathbb{Z}$. All output words have length in $O(|u| + |v|)$ and the algorithm runs in time $O(|u| + |v|)$.
\end{proposition}

\begin{proof}
Start by replacing $u$ and $v$ by their shortlex reductions $\pi(u)$, $\pi(v)$.
Let $u' := \pi((u_C)^M)$ and $v' := \pi((v_C)^M)$.
Then $|u'_C|_G > 2L$ and $|v'_C|_G > 2L$ so applying Proposition \ref{prop:eh_solve_conj} yields words $a'$ and $y'$ and a set $S'$ of words,
such that $y'$ is shortlex straight and $u'^{g'} =_G v'$ implies that $g' := a'y'^ns'$ for some $s' \in S'$. If $u^g =_G v$ then $u'^{u_L^{-1}g} =_G v'^{v_L^{-1}}$ so $u_L^{-1}gv_L =_G a'y'^ns'$ for some $s' \in S'$ and, after re-arranging, $g =_G u_La'y'^ns'v_L^{-1}$.

It suffices, then, to take $a := u_La'$, $y := y'$ and $S := \{ s'v_L^{-1} : s' \in S' \}$. Since $M$ is in $O(1)$, finding these values takes time $O(|u'| + |v'|) = O(|u| + |v|)$ and the proposition is proved.
\end{proof}

\begin{corollary}
\label{corollary:test_fin_ord}
There is an algorithm \textsc{TestInfOrder} that runs in time $O(|w|)$,
which tests whether an input word $w$ has infinite order.
\end{corollary}

\begin{proof}
First replace $w$ with $\slex(w)$. Now if $|(\slex((w_C)^M))_C|_G > 2L$ then
$(w_C)^M$ and therefore $w$ is of infinite order by Proposition \ref{prop:short_conj} and the algorithm returns \textsc{True}.
If not, $w$ cannot be of infinite order by Proposition \ref{prop:make_long} or Proposition \ref{prop:long_stays_long}, and the algorithm returns \textsc{False}.
Since $|(w_C)^M| = M|w|$, this test takes time at worst $O(|w|)$.
\end{proof}

Recall that our aim is to test the lists $A = (a_1, \ldots, a_m)$ and
$B = (b_1, \ldots, b_m)$ of words for conjugacy in $G$, and we are assuming
in this section that $a_1$ has infinite order. By the above corollary, we may
assume also that $b_1$ has infinite order, since otherwise $A$ and $B$ cannot
be conjugate.

By applying Proposition \ref{prop:conj_candidates} to $a_1$ and $b_1$,
and then replacing $A$ by $A^a$ and $B$ by $B^{s^{-1}}$ for each $s$ in
turn, we may effectively assume that the conjugating element has the form $y^n$.
This motivates the next subsection, which investigates the conjugation of
single words by straight powers.

\subsection{Conjugating by a power of a straight word}

In this subsection, suppose that geodesic words $g$ and $y$ are given, and that $y$ is straight. The aim is to find a description of the conjugates $g^{y^n}$ that allows, for any $g' \in G$, those values $n \in \mathbb{Z}$
for which $g' =_G g^{y^n}$ to be efficiently found.

The following preliminary result is true of general hyperbolic graphs, and will be specialised to the situation described above afterwards.

\begin{figure}
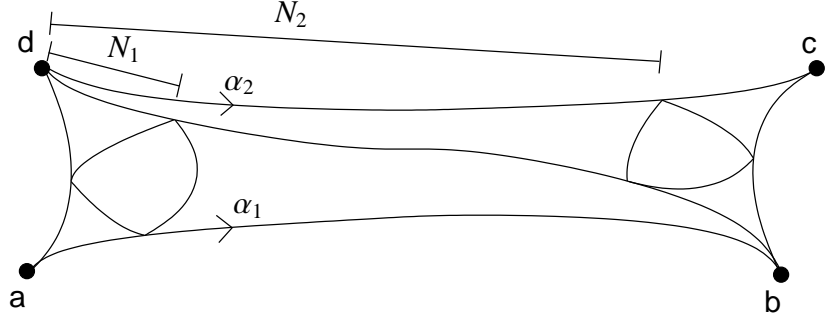

\begin{center}
\input thin_quad.pstex_t
\caption{A geodesic quadrilateral}
\label{figure:thin_quad}
\end{center}
\end{figure}

\begin{figure}
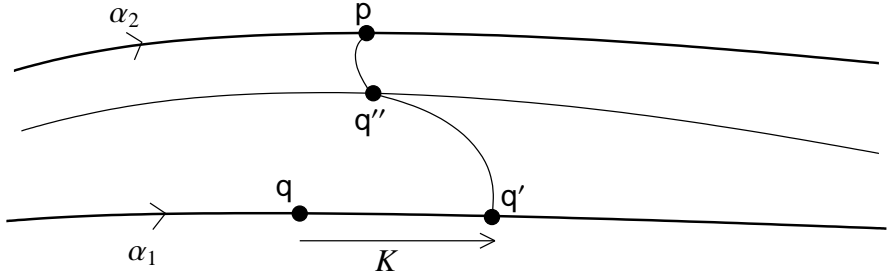

\begin{center}
\input thin_quad_zoom.pstex_t
\caption{A thin part of a quadrilateral}
\label{figure:thin_quad_zoom}
\end{center}
\end{figure}

\begin{figure}
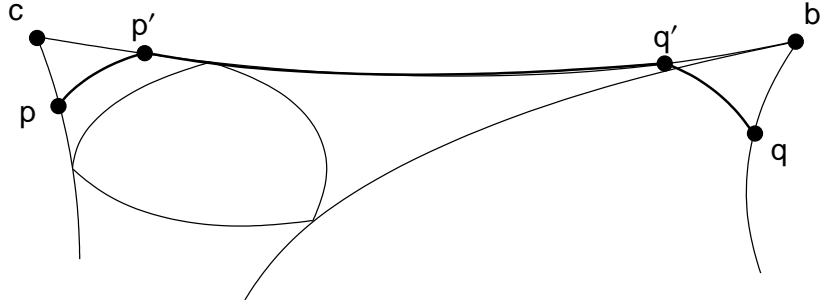

\begin{center}
\input quad_end.pstex_t
\caption{Points after the meeting points are distant}
\label{figure:quad_end}
\end{center}
\end{figure}

\begin{lemma}
\label{lemma:trapezium}
Let $\vertex{a}$, $\vertex{b}$, $\vertex{c}$ and $\vertex{d}$ be vertices in $\Gamma$ such that $l:= d(\vertex{a}, \vertex{b}) = d(\vertex{c}, \vertex{d})$. Let $\alpha_1 : [0, l] \to \Gamma$ be a geodesic path from $\vertex{a}$ to $\vertex{b}$ and let $\alpha_2 : [0, l] \to \Gamma$ be a geodesic path from $\vertex{d}$ to $\vertex{c}$ as in Figure \ref{figure:thin_quad}.

Define the constants
\[
K := d(\vertex{a}, \vertex{b}) - d(\vertex{b}, \vertex{d}),\ 
N_1  := (\vertex{a}, \vertex{b})_{\vertex{d}},\ 
N_2 :=  (\vertex{b}, \vertex{c})_{\vertex{d}}.
\]

Then, for $i \ge 0$ we have:
\begin{enumerate}
\item If $N_1 \le i \le N_2$ then $d(\alpha_2(i), \alpha_1(i+K)) \le 2\delta.$
\item If $N_1 + K \le i \le N_2 + K$ then $ d(\alpha_2(i-K), \alpha_1(i)) \le 2\delta. $
\item If $l \ge i \ge \max\{N_1 + K, N_2, N_2 +K\}$ then $ d(\alpha_1(i), \alpha_2(i)) =^{3\delta} d(\vertex{b}, \vertex{c}) - 2 (l - i). $
\end{enumerate}

Furthermore, if $l \ge i \ge d(\vertex{a}, \vertex{d})$ then at least one of these three cases applies.
\end{lemma}

\begin{proof}
Let $\gamma := \path{\vertex{b}}{\vertex{d}}$ be a geodesic, so that there are two geodesic triangles sharing the common side $\gamma$,
one with corners $\vertex{a}$, $\vertex{b}$, and $\vertex{d}$,
and the other with corners $\vertex{b}$, $\vertex{d}$ and $\vertex{c}$.
Also, let $\vertex{p} := \alpha_2(i)$ and $\vertex{q} := \alpha_1(i)$.

Suppose that $N_1 \le i \le N_2$. Then $\vertex{p}$ corresponds to some point $\vertex{q''}$ on $\gamma$ which in turn corresponds to some point $\vertex{q'}$ on $\alpha_1$ as illustrated in Figure \ref{figure:thin_quad_zoom}.
Observe that
\begin{eqnarray*}
      d(\vertex{a}, \vertex{q'})
& = & d(\vertex{a}, \vertex{b}) - d(\vertex{b}, \vertex{q'})
 =  d(\vertex{a}, \vertex{b}) - d(\vertex{b}, \vertex{q''})
= d(\vertex{a}, \vertex{b}) - d(\vertex{b}, \vertex{d}) + d(\vertex{d}, \vertex{q''}) \\
& = & d(\vertex{a}, \vertex{b}) - d(\vertex{b}, \vertex{d}) + d(\vertex{d}, \vertex{p}) 
 =  K + d(\vertex{d}, \vertex{p})
= K + i \\
& =& K + d(\vertex{a}, \vertex{q}),
\end{eqnarray*} so $\vertex{q'} = \alpha_1(i+K)$, and a geodesic path between $\vertex{p}$ and $\vertex{q'}$ has length at most $2\delta$ as required in the first case. 

For the second case, just use the first case with $i-K$ in place of $i$.

For the final case, note that
\begin{align}
    N_1 + K 
& = \frac{d(\vertex{d}, \vertex{a}) + d(\vertex{d}, \vertex{b}) - d(\vertex{a}, \vertex{b})}{2} + d(\vertex{a}, \vertex{b}) - d(\vertex{b}, \vertex{d}) \notag\\
& = \frac{d(\vertex{a}, \vertex{d}) + d(\vertex{a}, \vertex{b}) - d(\vertex{b}, \vertex{d})}{2} \notag\\
& = (\vertex{b}, \vertex{d})_{\vertex{a}}, \tag{$\ast$}
\end{align} the distance from $\vertex{a}$ to the meeting point on $\alpha_1$.

Now suppose that $l \ge i \ge \max\{N_1 + K, N_2, N_2 + K\}$. Let $\beta$ be a geodesic from $\vertex{b}$ to $\vertex{c}$. Then $d(\vertex{d}, \vertex{p}) \ge N_2$, so $\vertex{p}$ corresponds to a vertex $\vertex{p'}$ on $\beta$. Similarly, $d(\vertex{a},\vertex{q}) \ge N_1+K=(\vertex{b},\vertex{d})_{\vertex{a}}$ by ($\ast$) so $\vertex{q}$ corresponds to a vertex $\vertex{q''}$ on $\gamma$ with $d(\vertex{d}, \vertex{q''}) = i - K \ge N_2$, which in turn corresponds to a vertex $\vertex{q'}$ on $\beta$. This is illustrated in Figure \ref{figure:quad_end}.

Now,
\begin{eqnarray*}
      d(\vertex{p'}, \vertex{q'})
& = & d(\vertex{b}, \vertex{p'}) - d(\vertex{b}, \vertex{q'})
 = d(\vertex{b}, \vertex{c}) - d(\vertex{c}, \vertex{p}) - d(\vertex{b}, \vertex{q'})\\
 &=& d(\vertex{b}, \vertex{c}) - d(\vertex{b}, \vertex{q}) - d(\vertex{b}, \vertex{q})
= d(\vertex{b}, \vertex{c}) - 2d(\vertex{b}, \vertex{q}) 
 =  d(\vertex{b}, \vertex{c}) - 2(l - i),
\end{eqnarray*}
so $d(\alpha_1(i), \alpha_2(i)) =^{3\delta} d(\vertex{b}, \vertex{c}) - 2(l-i)$ as required.

For the final statement, assume that $i \ge d(\vertex{a}, \vertex{d})$ and that the first two cases do not apply. Since $i \ge d(\vertex{a}, \vertex{d}) \ge (\vertex{a}, \vertex{b})_{\vertex{d}} = N_1$, either $i > N_2$ or Case $1$ applies. Similarly, ($\ast$) implies that $i \ge d(\vertex{a}, \vertex{d}) \ge (\vertex{b}, \vertex{d})_{\vertex{a}} = N_1 + K$, so $i > N_2 + K$ or Case $2$ applies. Therefore $i \ge \max\{N_1+K, N_2, N_2 + K\}$ and Case $3$ applies.
\end{proof}

This lemma enables us to prove some results about the conjugates $g^{y^n}$ studied in this subsection. In particular, using the construction above in the group for some large power of $y$ provides computable estimates on the lengths of all conjugates by smaller powers of $y$, and also a constraint
on the form of those conjugates that are short in $G$.
For the remainder of this section, the shorthand $\canc{u}{v} = (\evaluate{u}, \evaluate{v})_{\vertex{1}}$ is adopted for words $u$ and $v$.

\begin{lemma}
\label{lemma:long_rectangle}
Suppose that $y$ is a straight word and that $g$ is a geodesic word.
Let $n \ge 0$, let $K := |y|n - |gy^n|_G$ and let $0 \le j \le n$.

\begin{enumerate}
\item If $\canc{g}{gy^n} \le |y|j \le \canc{gy^n}{y^n}$ then $g^{y^j} =_G h(y^n(K))^{-1}$ for some word $h$ with $|h| \le 2\delta$. 
\item If $\canc{g}{gy^n} + K \le |y|j \le \canc{gy^n}{y^n} + K$ then $g^{y^j} =_G y^{-n}(K)h$ for some word $h$ with $|h| \le 2\delta$. 
\item If $|y|n \ge |y|j \ge \max \{\canc{gy^n}{y^n}, \canc{g}{gy^n} + K, \canc{gy^n}{y^n} + K\}$ then $|g^{y^j}|_G =^{3\delta} |g^{y^n}|_G - 2|y|(n -j)$.
\end{enumerate}

Furthermore, if $|y|j \ge |g|$ then at least one of these three cases applies.
\end{lemma}

\begin{proof}
Let $\vertex{a} := \evaluate{g}$, $\vertex{b} := \evaluate{gy^n}$, $\vertex{c} := \evaluate{y^n}$ and $\vertex{d} := \vertex{1}$, and note that the three cases of Lemma \ref{lemma:trapezium} (with $i = |y| j$) correspond exactly to the three cases here. Notice that $\evaluate{gy^n(k)} = \alpha_1(k)$ and $\evaluate{y^n(k)} = \alpha_2(k)$ for each $k$.

In the first case, $d(\evaluate{y^n(i)}, \evaluate{gy^n(i + K)}) \le 2\delta$ so there is a word $h$ of length at most $2\delta$ with $\vertex{d} \cdot y^n(i)h = \vertex{a} \cdot y^n(i + K)$. By definition, $y^n(i) = y^j$ and $y^n(i+K) =_G y^jy^n(K)$. Now, $g^{y^j}$ labels a path from $\evaluate{gy^n(i)}$ to $\evaluate{y^n(i)}$ so $g^{y^j} =_G h(y^n(K))^{-1}$ as required.

For the second case, $y^n(i - K) =_G y^jy^{-n}(K)$ so by a similar argument $g^{y^j} =_G y^{-n} (K) h$ for some word $h$ of length at most $2\delta$ as required.

For the third case, since $d(\vertex{b}, \vertex{c}) = |g^{y^n}|_G$ and $d(\vertex{a}, \vertex{b}) = |y|n$, the result follows from the third part of Lemma \ref{lemma:trapezium}.

Noting that $|g| = d(\vertex{a}, \vertex{d})$, the final statement again corresponds to the final statement of Lemma \ref{lemma:trapezium}.
\end{proof}

Recall that the aim is to find a convenient description of the conjugates $g^{y^n}$. The first step will be to determine whether a power of $y$ centralises $g$,
and thus establish whether the set of conjugates is infinite.

Since the conjugates in the first range in Lemma \ref{lemma:long_rectangle} are parametrised by a word of length at most $2\delta$, if a large number of $j$ in this range can be found, some conjugate will repeat and some power of $y$ will indeed be in the centraliser of $g$. The next lemma states this more precisely.

\begin{lemma}
\label{lemma:thinness_extends}
Suppose that $y$ is a straight word, $g$ is a geodesic word, and
$n \in \mathbb{Z}$ with $n \ge 0$.
If $n - \floor{\frac{|g| + |g^{y^n}|_G}{2|y|}} > V$
then there exist constants $d, e$ with $|g| - 2\delta \le d \le |g|$ and
$1 \le e \le V$ such that $ y^e \in C_G(g)$ and
$ |g^{y^k}|_G =^{2\delta} d $ for all $k \in \mathbb{Z}$.
\end{lemma}

\begin{proof}
The number of $j$ that satisfy the first case of
Lemma \ref{lemma:long_rectangle} is at least
\begin{eqnarray*}
 \ceil{\frac{\canc{gy^n}{y^n} - \canc{g}{gy^n}}{|y|}} 
 &=&  \ceil{\frac{|gy^n|_G + |y|n - |g^{y^n}|_G}{2|y|} - \frac{|g| + |gy^n|_G - |y|n}{2|y|}} \\
& = & \ceil{\frac{2|y|n - |g^{y^n}|_G - |g|}{2|y|}} 
 =  n - \floor{\frac{|g| + |g^{y^n}|_G}{2|y|}}.
\end{eqnarray*}

Since the conjugates $g^{y^j}$ for such values of $j$ are all of the form
$h(y^\infty(K))^{-1}$ for words $h \in B_{2\delta}(1)$,
if there are more than $V$ such $j$, then there must exist $i,j$ with $i<j$,
$g^{y^i} =_G g^{y^j}$ and $e:=j-i\le V$.
So $y^e$ is in the centraliser of $g$, as required.

This implies that each conjugate $g^{y^k}$ is equal to some $g^{y^j}$ where $j$
satisfies the first case of Lemma \ref{lemma:long_rectangle},
so $g^{y^k} =_G h(y^\infty(K))^{-1}$ with $h \in B_{2\delta}(1)$.
Hence $|g^{y^k}|_G =^{2\delta} |K|$ for all $k$,  and putting $k=0$ gives
$|g| \le |K| + 2\delta$. Finally, $|K| = \big||y|n - |gy^n|_G\big| \le |g|$,
so taking $d = |K|$ completes the proof.
\end{proof}

The following lemma illustrates that testing whether some power of $y$ is in
the centraliser of $g$ is as simple as finding the length of a single word.

\begin{lemma}
\label{lemma:test_centraliser}
Suppose that $y$ is a straight word and that $g$ is a geodesic word, and
let $N \in \mathbb{Z}$ with $N > V + \floor{\frac{|g| + \delta}{|y|}}$. Then:\\
(i) if $|g^{y^N}|_G \le |g| + 2\delta$ then
   $N - \floor{\frac{|g| + |g^{y^N}|_G}{2|y|}} > V$;\\
(ii)  $|g^{y^N}|_G \le |g| + 2\delta$ if and only if some power of $y$
centralises $g$.
\end{lemma}

\begin{proof}
The first part is just straightforward evaluation: \begin{eqnarray*}
        N - \floor{\frac{|g| + |g^{y^N}|_G}{2|y|}}
&  >  & V + \floor{\frac{|g| + \delta}{|y|}} - \floor{\frac{|g| + |g^{y^N}|_G}{2|y|}} \\
& \ge & V + \floor{\frac{|g| + \delta}{|y|}} - \floor{\frac{2|g| + 2\delta}{2|y|}} 
 = V.
\end{eqnarray*}

For the second part, note that the first part covers the ``only if'' case by
Lemma \ref{lemma:thinness_extends}, so it remains to prove the ``if'' case.
Suppose that $y^e \in C_G(g)$ for some $e > 0$, and let $N_1:=e(V + |g| + 1)$.
Clearly $y^{N_1} \in C_G(g)$, so in particular $|g^{y^{N_1}}|_G = |g| \le |g| + 2 \delta$. Also
\[
        N_1 - \floor{\frac{|g|+|g^{y^{N_1}}|_G}{2|y|}}
  =   N_1 - \floor{\frac{2|g|}{2|y|}}
\ge eV + |g|e + e - |g| >  V,
\]
so Lemma \ref{lemma:thinness_extends} implies $|g^{y^k}|_G \le |g| + 2\delta$
for all $k \in \mathbb{Z}$.
\end{proof}

It remains to analyse the behaviour of the conjugates when no power of $y$ centralises $g$. The next lemma shows that the length of conjugates $g^{y^n}$ for large $n$ is predictable in this situation.

\begin{lemma}
\label{lemma:unbd_conj_long}
Suppose that $y$ is a straight word and that $g$ labels a geodesic in $\Gamma$. If $N > \frac{|g|}{|y|}$ and $|g^{y^N}|_G > |g| + 2\delta$ then $|g^{y^n}|_G =^{3\delta} |g^{y^N}|_G + 2|y|(n - N)$ for all $n \ge N$.
\end{lemma}

\begin{proof}
Apply Lemma \ref{lemma:long_rectangle} with $j = N$. Since $N|y| > |g|$,
at least one of the three cases applies. Because
$|g^{y^N}|_G > |g| + 2\delta \ge K + 2\delta$, the conclusions of the
first two cases cannot apply. So the third case must apply and
$|g^{y^N}|_G =^{3\delta} |g^{y^n}|_G - 2|y|(n-N)$,
which implies the required equation.
\end{proof}

The next result is simply a summary of the above results.

\begin{proposition}
\label{proposition:large_power_conjugate}
Let $g \in G$ and let $y$ be some straight word.
Let $N > V + \floor{\frac{|g|_G + \delta}{|y|}}$.
Then one of the following is true:

\begin{enumerate}
\item \label{case:lpc:central} $|g^{y^N}|_G \le |g|_G + 2\delta$ and there is
some $0 < e \le V$ such that $y^e \in C_G(g)$.
\item \label{case:lpc:noncentral} $|g^{y^N}|_G > |g|_G + 2\delta$ and
$|g^{y^n}|_G =^{3\delta}|g^{y^N}|_G + 2|y|(n-N)$ for all $n \ge N$. 
\end{enumerate}
\end{proposition}

Given words $u$ and $v$ and a shortlex straight word $y$, the preceding
proposition can be used to test whether $u^{y^n} =_G v$ for some integer $n$.

\begin{proposition}
\label{prop:test_conj_vs_sls}
Let $u, v$ be words and let $y$ be a straight word. In time $O(|u|+|v|+|y|)$ it is possible to find $r, t \in \mathbb{Z} \cup \{\infty\}$ such that either 

\begin{enumerate}
\item \label{case:pl:central}  $0\le r<t\le V$ and $u^{y^j} =_G v$ if and only if $j \equiv r \mod t$;
\item \label{case:pl:conj} $r \in \mathbb{Z}$, $t=\infty$ and $r$ is the unique integer such that $u^{y^r} =_G v$; or
\item \label{case:pl:not_conj} $r=\infty$, $t=\infty$ and there is no integer $n$ such that $u^{y^n} =_G v$.
\end{enumerate}
\end{proposition}

\begin{proof} 
First, let $N:= V + 1 + \floor{\frac{|u|_G + |v|_G + \delta}{|y|}}$ and let $l_g:=|g^{y^N}|_G$, where $g$ is either $u$ or $v$.

If $l_u \le |u|_G + 2\delta$ but $l_v > |v|_G + 2\delta$ then by Proposition \ref{proposition:large_power_conjugate}, the conjugates $u^{y^n}$ have bounded length whereas the conjugates $v^{y^n}$ do not. Thus there can be no $n \in \mathbb{Z}$ such that $u^{y^n} =_G v$. The same is true if these two inequalities are reversed, so if $u$ and $v$ lie in different cases of Proposition \ref{proposition:large_power_conjugate} then set $r = t = \infty$ and stop. Otherwise, both $u$ and $v$ lie in the same case of Proposition \ref{proposition:large_power_conjugate}.

Suppose that $l_u \le |u|_G + 2\delta$. By Proposition \ref{proposition:large_power_conjugate},
some power $y^e$ with  $0 < e \le V$ centralises $u$,
so in particular Case \ref{case:pl:conj} does not apply. Since $V$ is bounded above in terms of $|X|$ and $\delta$, it is possible to check for each $0 \le r' < t' \le V$ if $u^{y^{t'}} =_G u$ or $u^{y^{r'}} =_G v$ in time $O(|u|+|v|+|y|)$. If no $r'$ is found, Case 3 holds so let $r = t = \infty$. Otherwise Case 1 holds so pick the lowest values found for $r'$ and $t'$ as $r$ and $t$ respectively.

Finally, suppose that $l_u > |u|_G + 2\delta$. Proposition \ref{proposition:large_power_conjugate} implies that $|u^{y^n}|_G =^{3\delta} l_u + 2|y|(n-N)$ for large $n$, so Case \ref{case:pl:central} cannot apply and no power of $y$ is in the centraliser of $u$. In fact, by Proposition  \ref{proposition:large_power_conjugate}, if $u^{y^r} =_G v$ then, for all sufficiently large $n$,
\[
l_u +2|y|(n+r-N) 
 =^{3\delta}  |u^{y^{n+r}}|_G
 = |v^{y^n}|_G
=^{3\delta} l_v + 2|y|(n-N).
\]
Rearranging, $l_v-l_u=^{6\delta}2|y|r$, so $\frac{l_v-l_u-6\delta}{2|y|} \le r \le \frac{l_v-l_u+6\delta}{2|y|}$. Because no power of $y$ centralises $u$, there can only be one $n$ such that $u^{y^n} =_G v$ and to find it, we must simply check each $r$ in this range. If some $y^r$ conjugates $u$ to $v$ then Case 2 holds so set $t=\infty$ and stop, otherwise Case 3 holds so set $r=t=\infty$. At most $6\delta+1$ checks of conjugates $u^{y^n}$ need to made to distinguish between these two cases, and each check takes time $O(|u|+|v|+|y|)$ as required. 
\end{proof}

\subsection{Testing conjugacy of $A$ and $B$}

Recall that $A = (a_1, \ldots, a_m)$ and $B = (b_1, \ldots, b_m)$, that $a_1$
has infinite order, and the aim is to test if there is an element $g\in G$ with
$A^g =_G B$. We can now present an algorithm to carry out this test.
Furthermore, it will find the set of all $g \in G$ with this property.
Let $\mu$ be an upper bound on the length of elements in the two lists.

Use Corollary \ref{corollary:test_fin_ord} to test in time $O(|b_1|)$ if $b_1$ is of infinite order. If it is not, $a_1$ and $b_1$ are not conjugate, so neither are $A$ and $B$ and the algorithm returns \textsc{False}.

Next, apply Proposition \ref{prop:conj_candidates} to $a_1$ and $b_1$ to obtain
a word $p$, a shortlex straight word $y$ and a set $S$ of at most $V$ words
such that $a_1^g =_G b_1$ only if $g =_G py^ns$ for some $n \in \mathbb{Z}$ and $s \in S$. All returned words have length $O(|a_1| + |b_1|)$ and this step
takes time $O(|a_1| + |b_1|) \le O(\mu)$.

The following steps are carried out for each $s \in S$. Since $|S| \le V$,
it is sufficient to show that the time taken is $O(m\mu)$ for each $s \in S$.

For each $i \in \{1, \ldots, m\}$, applying
Proposition \ref{prop:test_conj_vs_sls} to $a_i^p$, $b_i^{s^{-1}}$ and $y$
provides values $r_i$ and $t_i$ with $a_i^{py^{r_i+jt_i}} =_G b_i^{s^{-1}}$
for all $j \in \mathbb{Z}$ in time $O(m\mu)$.

If $r_i = \infty$ for some $i$ then $a_i^p$ is not conjugated to $b_i^{s^{-1}}$
by any power of $y$, so the same is true of $A$ and $B$, and we delete $s$ from
$S$.

Otherwise, if $t_i = \infty$ for some $i$, then $y^{r_i}$ is the only
power of $y$ that might conjugate $A^p$ to $B^{s^{-1}}$.  So we test whether
this is the case.  If so, then we set $T_s :=0$ and
$R_s := r_i$. If not, then we delete $s$ from $S$.

The remaining case is where all $t_i$ and $r_i$ are finite, in which case the set of equations $j \equiv r_i \mod t_i$ must be solved simultaneously.
By the Chinese Remainder Theorem, there is either no solution to these
equations, or the set of solutions has the form
$\{R_s + nT_s \mid n \in \mathbb{Z}\}$,
where $T_s$ is the least common multiple of the $t_i$. 
Since $t_i \le V$ for all $i$, we have $T_s \le V!$, so we can
test whether there is a solution and, if so find $R_s$ and $T_s$,
in time $O(m)$.
If there is no solution, then we delete $s$ from $S$.

After carrying out the above computations for each $s \in S$,
we have a complete description of the set of elements $g \in G$ for which
$A^g =_G B$ has been obtained: they are precisely those elements $g =_G py^{R_s + nT_s}s$ for $s \in S$ and $n \in \mathbb{Z}$.

If $S$ is empty, then return \textsc{False}.
Otherwise return \textsc{True} and the conjugating element $py^{R_s}s$.
This completes the proof of Theorem 1 under the assumption that $a_1$
has infinite order.

\subsection{Finding the centraliser of $A$}

Let $B = A$ and proceed exactly as in the previous subsection,
except for the final paragraph.
The algorithm has established that all elements $g$ with $A^g =_G A$ are of
the form $py^{R_s + nT_s}s$ for some $s \in S$ and $n \in \mathbb{Z}$ and all
elements of this form are in $C_G(A)$.
It remains to find a finite generating set for $C_G(A)$.

If $T_s = 0$ for all $s \in S$, then $C_G(A)$ is finite and the algorithm
returns $\{ py^{R_s}s : s \in S \}$ as a generating set.

Otherwise, $T_s > 0$ for some $s \in S$.
Since $py^{R_s}s$ and $py^{R_s + T_s}s$ are both elements of the centraliser,
so is $py^{R_s + T_s}s(py^{R_s}s)^{-1} =_G py^{T_s}p^{-1}$.
Now, for $s,t \in S$ with $T_t>0$, we have
$(py^{T_s}p^{-1})^n(py^{R_t}t) =_G py^{R_t + nT_s}t$ for all $n \in \mathbb{Z}$,
so $T_t$ divides $T_s$ and hence all nonzero $T_s$ have the same value, $T$.
Noting that $(py^{T}p^{-1})^{-n}(py^{R_s + nT}s) =_G py^{R_s}s$ for any
$s \in S$ and $n \in \mathbb{Z}$, we see that $C_G(A)$ is generated by the set
$\{ py^{R_s}s : s \in S \} \cup \{ py^{T}p^{-1} \}$. This set has size in $O(1)$
and each element has length $O(\mu)$, so it can be computed in time $O(\mu)$.
This completes the proof of Theorem 2 under the assumption that $a_1$
has infinite order.

\section{Conjugacy of general lists} 

The purpose of this section is to show that the conjugacy problem for finite
lists is solvable in linear time even when all elements of both lists have
finite order. To do this, we either find an infinite order element that is a
product of some of the elements in one of the lists, or we reduce the problem
to the case in which both the length of the lists and the lengths of the
elements in the lists are bounded by a constant.

\subsection{Simple results}

We start with two elementary observations. A {\em mid-vertex} on a path is
defined to be a vertex at distance at most $1/2$ from the mid-point of the path.

\begin{lemma}
\label{lemma:triangle}
Suppose that $\vertex{x}$, $\vertex{y}$ and $\vertex{z}$ are vertices in
$\Gamma$ and $\vertex{p}$ is a mid-vertex of a geodesic path $\path{\vertex{x}}{\vertex{y}}$. Then \[
d(\vertex{p}, \vertex{z}) \le \frac{2\max\{d(\vertex{x}, \vertex{z}), d(\vertex{y}, \vertex{z})\} - d(\vertex{x}, \vertex{y}) + 1}{2} + \delta.
\]
\end{lemma}

\begin{proof}
If $\vertex{p}$ corresponds to a vertex $\vertex{q}$ on
$\path{\vertex{x}}{\vertex{z}}$, then $d(\vertex{q}, \vertex{z}) \le
d(\vertex{x}, \vertex{z}) - \frac{d(\vertex{x}, \vertex{y}) - 1}{2}$ so
$d(\vertex{p}, \vertex{z}) \le \frac{2d(\vertex{x}, \vertex{z}) -
d(\vertex{x}, \vertex{y}) + 1}{2} + \delta$.
Similarly, if $\vertex{p}$ corresponds to $\vertex{q}$ on
$\path{\vertex{y}}{\vertex{z}}$, then $d(\vertex{p}, \vertex{z}) \le \frac{2d(\vertex{y}, \vertex{z}) - d(\vertex{x}, \vertex{y}) + 1}{2} + \delta$. The
result follows.
\end{proof}

\begin{lemma} \label{lemma:replace}
Suppose $g, a_1, a_2, b_1, b_2 \in G$. Then $(a_1, a_2)^g = (b_1, b_2)$ if and only if $(a_1a_2, a_2)^g = (b_1b_2, b_2)$.
\end{lemma}

\subsection{Bounding element length in short lists}

This subsection is devoted to the proof of the following result.

\begin{proposition}
\label{prop:shorten_list}
There is an algorithm \textsc{ShortenWords} which, given a list
$A=(a_1,\ldots,a_m)$ of words, either:
\begin{itemize}
\item returns $c \in G$ such that, for all $1 \le i \le m$,
\[ |c^{-1}a_i a_{i+1} \cdots a_mc|_G \le
3^{m-i}\left(7L + \delta + \frac{1}{2}\right) \] or
\item returns integers $j$ and $k$ such that $1 \le j \le k \le m$ and
$a_j a_{j+1} \cdots a_k$ is of infinite order.
\end{itemize}

This algorithm runs in time $O(m^3\mu)$, where $\mu$ is the maximum length of the elements in $A$.
\end{proposition}

\begin{proof}
The algorithm is presented below. The remainder of the proof will be devoted to proving that it works as claimed.

\begin{algorithmic}[1]
\label{algorithm:shorten_list}
\Function{ShortenWords}{$[a_1, \ldots, a_m]$}
\State $c_0 \gets 1$
\For{$k:=1$ to $m$}
	\For{$j \in \{1, \ldots, k\}$}
		\If{$|(\pi(c_{k-1}^{-1}a_j \cdots a_kc_{k-1}))_C|_G > 2L$} \label{line:conj_finite_check}
			\State \textbf{return} ${\sf null}, j, k$ \Comment{$a_j\cdots a_k$ is of infinite order} \label{line:conj_finite_return_infinite}
		\EndIf
	\EndFor
	\State $c_k \gets \pi(c_{k-1}(\pi(c_{k-1}^{-1}a_kc_{k-1}))_L)$ \label{line:conj_finite_extend}
\EndFor
\State \textbf{return} $c_m, {\sf null}, {\sf null}$
\EndFunction
\end{algorithmic}

If the function finds and returns integers $j, k$ on Line \ref{line:conj_finite_return_infinite}, then a conjugate $g$ of $a_j \cdots a_k$ satisfies
$|\slex(g)_C| > 2L$, and so $g$ is of infinite order by Proposition \ref{prop:short_conj}. But then $a_j \cdots a_k$ has infinite order also and the algorithm is correct to return $j, k$.
The condition on Line \ref{line:conj_finite_check} can therefore be assumed
always to fail.

We show first that
$|c_k| \le k(\frac{\mu}{2}+\delta+1)$. Consider a geodesic triangle with corners $\vertex{1}$, $\vertex{b}:= \evaluate{c_{k-1}}$ and $\vertex{c}:=\evaluate{a_kc_{k-1}}$. Pick shortlex reduced words to label the paths $\path{\vertex{1}}{\vertex{b}}$, $\path{\vertex{b}}{\vertex{c}}$ and $\path{\vertex{1}}{\vertex{c}}$. Let $\vertex{p}:=\evaluate{c_{k-1}(\slex(c_{k-1}^{-1}a_kc_{k-1}))_L}$, which is a mid-vertex of $\path{\vertex{b}}{\vertex{c}}$ as illustrated in Figure \ref{figure:extend_c}. Since $c_k$ is a geodesic from $\vertex{1}$ to $\vertex{p}$,
we have by Lemma \ref{lemma:triangle}
\begin{eqnarray*}
        |c_k|
& \le & \frac{2\max\{d(\vertex{1}, \vertex{b}), d(\vertex{1}, \vertex{c})\} - d(\vertex{b}, \vertex{c}) + 1}{2}+\delta \\
& \le & \frac{2\max\{|c_{k-1}|,|a_kc_{k-1}|_G\}-|c_{k-1}^{-1}a_kc_{k-1}|_G+1}{2}+\delta.
\end{eqnarray*}

\begin{figure}
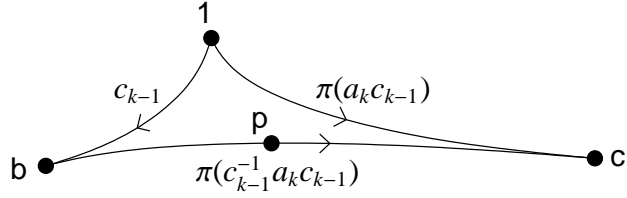

\begin{center}
\input extend_c.pstex_t
\caption{Extending $c$.}
\label{figure:extend_c}
\end{center}
\end{figure}

Suppose $|c_{k-1}| \ge |a_kc_{k-1}|_G$. Notice that $|c_{k-1}^{-1}a_kc_{k-1}|_G \ge |c_{k-1}| - |a_kc_{k-1}|_G$ by the triangle inequality, so
\begin{eqnarray*}
        |c_k|
& \le & \frac{2|c_{k-1}| - |c_{k-1}| + |a_kc_{k-1}|_G + 1}{2} + \delta
  =  \frac{|c_{k-1}| + |a_kc_{k-1}|_G + 1}{2} + \delta \\
& \le & \frac{2|c_{k-1}| + |a_k| + 1}{2} + \delta\
\le |c_{k-1}| + \frac{|a_k|}{2} + \delta + 1.
\end{eqnarray*}
Similarly, if $|c_{k-1}| < |a_kc_{k-1}|_G$ then
\begin{eqnarray*}
        |c_k|
& \le & \frac{2|a_kc_{k-1}|_G - |a_kc_{k-1}|_G + |c_{k-1}|+1}{2} + \delta
 =  \frac{|a_kc_{k-1}|_G + |c_{k-1}|+1}{2} + \delta \\
& \le & \frac{|a_k| + 2|c_{k-1}|+1}{2} + \delta
 \le  |c_{k-1}| + \frac{|a_k|}{2} + \delta + 1.
\end{eqnarray*}
So in either case $|c_k| \le |c_{k-1}| + \frac{|a_k|}{2} + \delta + 1$,
and induction on $k$ gives
$|c_k| \le k(\frac{\mu}{2} + \delta + 1)$, as required.

We can now show that the function completes in time $O(m^3\mu)$.
Note that
$$|c_{k-1}^{-1}a_j\cdots a_kc_{k-1}| \le k\mu+2|c_{k-1}| \le 2k(\mu + \delta + 1)$$
so the checks on Line \ref{line:conj_finite_check} each run in time $O(k\mu)$. There are $k$ such steps per loop and a total of $m$ loops, so the overall running time is in $O(m^3\mu)$ for this step.
Similarly, $|c_{k-1}c_{k-1}^{-1}a_kc_{k-1}| \in O(k\mu)$ so Line \ref{line:conj_finite_extend} runs in time $O(k\mu)$ and the overall time taken in this step is in $O(m^2\mu)$. Therefore the whole algorithm runs in time $O(m^3\mu)$ as required.

It remains to show that the bound on the length of the elements $(a_i\cdots a_m)^{c_m}$ is satisfied.
For each $k\in\{1,\ldots,m\}$, define $K_{k,k}:=2L$, and let
$K_{i, k+1}: = 3K_{i,k} + 10L + 2\delta + 1$ for $1 \le i\le k$.
We shall use induction on $k$ to show that $|c_k^{-1}a_i \cdots a_k c_k|_G \le K_{i, k}$ for any $1 \le i \le k$ and then show that $K_{i,m}$ is within the required bound.

In the case $k = i$, we have  $a_k^{c_k} =_G d^{d_L} =_F d_C$ where $d = \slex(a_k^{c_{k-1}})$. But then Line \ref{line:conj_finite_check} ensures that $|a_k^{c_k}|_G \le K_{k,k} = 2L$.

Now suppose that, for some $k$, the inequality
$|c_k^{-1}a_i\cdots a_k c_k|_G\le K_{i,k}$ is satisfied for all $1 \le i \le k$. Showing that $|c_{k+1}^{-1} a_i \ldots a_{k+1}c_{k+1}|_G \le K_{i,k+1}$ for each $i$ will complete the induction.

Pick some specific $i$, and let $e := \slex(c_k^{-1}a_i \ldots a_{k+1} c_k)$ and $g := \slex(c_k^{-1}a_{k+1}c_k)$. Notice that $c_{k+1} =_G c_kg_L$ and so
\[
        (a_i \cdots a_{k+1})^{c_{k+1}}
 =_G  e^{c_k^{-1}c_{k+1}}
 =_G  e^{g_L}
 =_G  (e_C)^{e_L^{-1}g_L}
 =_G  (e_C)^{e_L^{-1}g_R^{-1}g_C}.
\]

The checks on Line \ref{line:conj_finite_check} ensure that $|e_C|_G \le 2L$,
and $|g_C|_G \le 2L$, so we know that
$\left| (e_C)^{e_L^{-1}g_R^{-1}g_C} \right|_G \le 2|g_Re_L|_G + 6L$.
Hence the induction will be complete if it can be shown that \begin{equation} \label{eqn:short_grel}
|g_Re_L|_G \le \frac{3}{2}K_{i,k} + 2L + \delta + \frac{1}{2}. \end{equation}

Let $f := \slex(c_k^{-1}a_i\ldots a_k c_k) =_G eg^{-1} $ and recall that $|f|\le K_{i,k}$ by the inductive assumption. Consider a geodesic triangle with corners $\vertex{1}$, $\vertex{b} := \evaluate{g}$ and $\vertex{c} := \evaluate{ge_L}$ illustrated in Figure \ref{figure:shorten_tri}. Note that
\[
        d(\vertex{1}, \vertex{c})
  =   |ge_L|_G \\
  =   |f^{-1} ee_L|_G \\
 \le  |ee_L|_G + K_{i,k}\\ 
  =   |e_Le_C|_G + K_{i,k},
\]
but $|e_C|_G \le 2L$ so
\[
        d(\vertex{1}, \vertex{c})
 \le  |e_L| + K_{i,k} + 2L \\
 \le  \frac{|e|}{2} + K_{i,k} + 2L \\
 \le  \frac{|f| + |g|}{2} + K_{i,k} + 2L.
\]

\begin{figure}
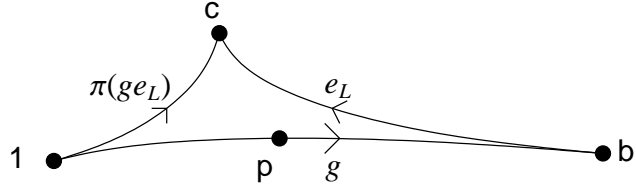

\begin{center}
\input shorten_tri.pstex_t
\caption{Bounding $g_Re_L$}
\label{figure:shorten_tri}
\end{center}
\end{figure}

Also, $d(\vertex{b}, \vertex{c}) = |e_L| \le \frac{|e|}{2} \le \frac{|f| + |g|}{2}$.

Pick the mid-vertex $\vertex{p} := \evaluate{g_L}$ on $\path{\vertex{1}}{\vertex{b}}$. Lemma \ref{lemma:triangle} implies that
\begin{eqnarray*}
        |g_R e_L|_G
&  =  & d(\vertex{p}, \vertex{c}) 
\le \frac{2 \max\{d(\vertex{1}, \vertex{c}), d(\vertex{b}, \vertex{c})\} -d(\vertex{1}, \vertex{b}) + 1}{2} + \delta \\
& \le & \frac{2 \max\{\frac{|f|+|g|}{2} + 2L + K_{i, k}, \frac{|f|+|g|}{2}\} -|g| + 1}{2} + \delta \\
&  =  & \frac{2(2L + K_{i, k}) + |g|+|f|-|g| + 1}{2} + \delta
 =  \frac{2(2L + K_{i, k}) + |f| + 1}{2} + \delta \\
& \le & \frac{3}{2}K_{i, k} + 2L + \delta + \frac{1}{2},
\end{eqnarray*} as required by (\ref{eqn:short_grel}).

This completes the proof that $|(a_i \cdots a_k)^{c_k}|_G \le K_{i,k}$
for each $1 \le i \le k \le m$, and to get the required bound on the length
of $(a_i\cdots a_m)^{c_m}$ it suffices to show that
$K_{i,k} \le 3^{k-i}(7L+\delta+\frac{1}{2})$, and then put $k=m$.
But, since $K_{i,k} = 3K_{i,k-1} + 10L + 2\delta + 1$,
a straightforward induction on $k$ starting at $k=i$ yields
$$K_{i,k} \le 
   3^{k-i} \times 2L + (3^{k-i}-1)\left(5L + \delta + \frac{1}{2}\right),$$
from which the required bound follows, and the proof is complete.
\end{proof}

Note that by repeated application of Lemma \ref{lemma:replace}, we see that
the conjugacy problems are equivalent for the lists $(a_1, \ldots, a_m)$ and
$(b_1, \ldots, b_m)$, and for the lists \linebreak $(a_1', a_2', \ldots, a_m')$ and
$(b_1', b_2', \ldots, b_m')$, where $a_i' = a_i \cdots a_m$ and
$b_i' = b_i \cdots b_m$.

\subsection{Some worse than linear time algorithms}

This subsection provides a toolbox of results that solve various problems
involving conjugacy and centralisers of lists in worse than linear time.
They are useful, as the previous subsection gives a method of bounding the lengths of elements in a list in terms of the number of elements.

\begin{proposition}\cite[Corollary 3.2]{bridson2005conjugacy}
\label{proposition:small_cent}
Let $(a_1, \ldots, a_m)$ be a list of words representing pairwise distinct
finite order elements of $G$. Suppose that $x \in G$ satisfies \[ |x|_G \ge (2k + 5)^{4\delta+2}(l + 2\delta) \] where $l = \max\{|a_1|_G, |a_1^x|_G, \ldots, |a_m|_G, |a_m^x|_G\}$ and $k$ is the number of generators of $G$.
Then $m \le V^4$.
\end{proposition}

The statement in \cite{bridson2005conjugacy} says that $m \le (2k)^{8 \delta}$,
but the proof there does in fact prove the statement here.
Proposition \ref{proposition:small_cent} implies that the centraliser of a long
list of distinct finite order elements is finite. Theorem III.$\Gamma$.3.2 of \cite{BriHae99} then provides a bound on the number of elements in a finite
subgroup:

\begin{proposition}
\label{proposition:finite_small_ball}
If $G$ is a $\delta$-hyperbolic group and $H$ is a finite subgroup of $G$ then there is an element $g \in G$ with $H^g$ contained entirely within a ball in the Cayley graph of $G$ of radius $4\delta+2$.
\end{proposition}


\begin{corollary}
\label{cor:short_cent}
There is a constant $R$ and an algorithm \textsc{FindCentraliserExp} that takes
as input a list $A$ consisting of $n > V^4$ words, all of which represent pairwise distinct finite order elements of $G$, returns the centraliser $C$ of $A$, and runs in time $O(n\mu R^\mu)$ where $\mu$ is an upper bound on the length of words in $A$. All elements of $C$ have length in $O(\mu)$ and the number of elements in $C$ is in $O(1)$.
\end{corollary}

\begin{proof} 
Suppose that $A = (a_1, \ldots, a_n)$ is such a list. If $x \in C$,
then $a_i^x = a_i$ for all $1 \le i \le n$, so $l = \mu$ in Proposition \ref{proposition:small_cent}. Hence $|x|_G < R(\mu + 2\delta)$, where $R := (2k + 5)^{4\delta+2}$, since $n > V^4$.

Since the elements in $C$ are of bounded length, $C$ is finite. Proposition \ref{proposition:finite_small_ball} implies that $C$ can be conjugated into a ball in $\Gamma$ of radius $4\delta + 2$, and in particular the number of elements in $C$ is bounded by a constant depending only on $G$.

Thus the algorithm \textsc{FindCentraliserExp} now just needs to check for each word $w$ of length at most $R(\mu+2\delta)$ whether $A^w =_G A$. There are at most $R^{\mu+2\delta} \in O(R^\mu)$ such words, and checking each word takes time $O(n\mu)$, so the algorithm runs in time $O(n\mu R^\mu)$ as required.
\end{proof}

Thus there is a method of computing the centraliser of a long list of finite
order words of bounded length, whose complexity is linear in the
length of the list.
Thus we can compute centralisers of lists of short elements.
The following result enables us to test conjugacy between lists of short
elements.

\begin{proposition}\cite[Theorem 3.3]{bridson2005conjugacy}
Let $A = (a_1, \ldots, a_m)$ and $B = (b_1, \ldots, b_m)$ be sets of
finite order elements in $G$. If $A$ and $B$ are conjugate then there exists a word $x$ with
$A^x=_G B$  and
\[ |x|_G \le (2k+5)^{4\delta+2}(\mu+2\delta) + V^{4V^4}, \]
where $\mu$ is the maximum length of an element in either list and $k$ is the number of generators of $G$.
\end{proposition}

Again, the statement in \cite{bridson2005conjugacy} uses $(2k)^{8 \delta}$ in place of $V^4$, but the proof is sufficient to prove the statement here.
Thus by simply checking each element up to the above bound on $|x|_G$, we have an
algorithm \textsc{TestConjugacyExp} that takes as input two lists of $m$ words whose elements have length less than $\mu$ and returns a word $w$ with $A^w =_G B$ if one exists in time exponential in $\mu$.

We shall also need an algorithm \textsc{FindCentraliserGenerators} that can be
used on an arbitrary list of finite order words.
In order to avoid defining the many concepts required while covering no new ground, the reader is referred to \cite{gersten1991rational} for a method of doing so even without the finite order condition:
Lemma 4.2 and Proposition 4.3 of \cite{gersten1991rational} show that the
centraliser $C$ of a finite list in a biautomatic group
(all hyperbolic groups are biautomatic) is a regular language and provide a
method of computing an automaton that accepts this language.
Theorem 2.2 of \cite{gersten1991rational} provides a proof that $C$ is then
quasiconvex and then Proposition 2.3 of \cite{gersten1991rational} provides an
explicit finite generating set for $C$. Each of these steps involves a
potentially exponential blow-up in space and time.
But we shall use \textsc{FindCentraliserGenerators} only with input
of bounded length, so it can be regarded as running in time $O(1)$.

\subsection{Ensuring distinct elements}

To apply Corollary \ref{cor:short_cent} to a list $A = (a_1, \ldots, a_m)$,
all of the elements of $A$ must represent distinct elements of $G$.
We shall eventually apply the corollary to a list of length at most $n=V^4 + 1$
that has been returned by \textsc{ShortenWords}, so it is necessary to ensure
that the words $\{a_i \cdots a_n \mid 1 \le i \le n \}$ represent distinct
group elements.

An algorithm \textsc{EnsureDistinct} will be used for this purpose.
It takes as input two lists of words $A=(a_1, \ldots, a_m)$ and
$B=(b_1, \ldots, b_m)$ and an integer $n \ge 1$.
It returns either \textsc{False} (in which case $A$ and $B$ cannot be conjugate
in $G$) or two lists $A'=(a'_1, \ldots, a'_{m'})$ and
$B'=(b'_1, \ldots, b'_{m'})$ with $m' \le m$, such that
\begin{enumerate}
\item For $g \in G$, $A^g=_G B$ if and only if $A'^g=_G B'$.
\item Let $n' = \min\{m',n\}$. Then the words
$\{ a_i \cdots a_{n'} \mid 1 \le i \le {n'} \}$ represent distinct elements
of $G$, as do the words $\{ b_i \cdots b_{n'} \mid 1 \le i \le {n'} \}$
\end{enumerate}

The algorithm works as follows. We start with $A':=A$, $B':=B$,
and then delete elements from $A'$ and $B'$ until Condition 2 holds,
while maintaining Condition 1.

To do this, consider the words $a'_{ij} := a'_ia'_{i+1}\cdots a'_j$
and $b'_{ij} := b'_ib'_{i+1}\cdots b'_j$ with $1 \le i \le j \le n$. 
Since  $A'^g=_G B'$ implies $a'^g_{ij} =_G  b'_{ij}$, if exactly one of $a'_{ij}$ and
$b'_{ij}$ is equal to the identity in $G$, then $A'$ and $B'$ cannot be
conjugate, so we return \textsc{False}.
If $a'_{ij} =_G 1$ and $b'_{ij} =_G 1$, then we delete $a'_j$ from $A'$ and
$b'_j$ from $B'$, which maintains Condition 1.

We continue to do this until none of the  elements $a'_{ij}$ and $b'_{ij}$ with
$1 \le i \le j \le n$ represent the identity of $G$, which implies that
Condition 2 holds, and we are done.

If $\mu$ is an upper bound on the lengths of the elements in the lists, then
we have to test at most $2mn$ elements of length at most $n\mu$ for being
the identity, so the algorithm runs in time $O(mn^2\mu)$.

\subsection{Solving the conjugacy and centraliser problems}

\label{section:solve_conj}

We can now complete the proofs of Theorems 1 and 2,
by describing the algorithms that solve the conjugacy and centraliser problems
with the required complexity. Since the algorithms are very similar, they will
be described together.

Let $A = (a_1, \ldots, a_m)$ and $B = (b_1, \ldots, b_m)$ be lists of words.
For the centraliser problem, set $B = A$.
For the conjugacy problem, we return either \textsc{False} or
an element of $g$ that conjugates $A$ to $B$.
For the centraliser problem, we return a finite generating set of $C_G(A)$.
Let $\mu$ be the maximum length of the words in $A$ and $B$.

We start by running
\textsc{EnsureDistinct}$(\slex(A), \slex(B), n)$ with $n:= \min(V^4+1,m)$.
If this returns \textsc{False}, then the lists are not conjugate so
return \textsc{False}.
Otherwise, replace $A$ and $B$ by the lists returned by \textsc{EnsureDistinct}.
Since $n$ is bounded, this step takes time $O(m\mu)$.

The two lists $A$ and $B$ now consist of shortlex reduced words,
such that, for $n := \min\{V^4+1, m\}$ (redefining $m$ to be the new length of $A$ and $B$, if necessary), the group elements represented by $a_i \cdots a_n$ are distinct for all $i \le n$.

Let $A'$ and $B'$ be the sublists of $A$ and $B$ respectively containing the
first $n$ elements.
Apply \textsc{ShortenWords} to $A'$ and $B'$; this takes time
$O(n^3\mu) = O(\mu)$.

\textsc{ShortenWords} may return an infinite order element $a_i \cdots a_j$
or $b_i \cdots b_j$ with $1 \le i \le j \le n$.
If not, then we set $j=n$ and run \textsc{TestInfOrder}($a_i \cdots a_n$) and
\textsc{TestInfOrder}($b_i \cdots b_n$) for $1 \le i \le n$, which takes
time $O(\mu)$. In either case if, for some $i,j$ we find one of
$a_i \cdots a_j$ or $b_i \cdots b_j$ has infinite order then we test
whether they both have infinite order and return \textsc{False} if not.

If we have found $i,j$ with $1 \le i \le j \le n$ such that
$a_i \cdots a_j$ and $b_i \cdots b_j$ both have infinite order, then we add
$a_i \cdots a_j$ to the start of $A$ and add $b_i \cdots b_j$ to the start of
$B$. This does not change the set of $g$ with $A^g=_G B$. It may increase
the maximum word length up to $n\mu$, but this remains in $O(\mu)$.
We can now apply the special cases of Theorems 1 and 2 proved in
Section~\ref{section:inf_order}, to complete the algorithms.

We may assume from now on that \textsc{ShortenWords} applied to
$A'$ and $B'$ does not return an infinite order element, and that
$a_i \cdots a_n$ and $b_i \cdots b_n$ have finite order for $1 \le i \le n$.
So \textsc{ShortenWords} returns conjugating elements $c_A$ and $c_B$.
We now (re)define
$A' := (a_1', \ldots, a_n')$ where $a_i' = \slex((a_i \cdots a_n)^{c_A})$
and define $B'$ in the same way using $c_B$.

Note that the total lengths of the elements in $A'$ and $B'$ are now in $O(1)$,
and hence all of our procedures will take time $O(1)$ when applied to $A'$ and
$B'$.

Use \textsc{TestConjugacyExp} to look for a word $u$ with $A'^u =_G B'$.
If no $u$ is found then return \textsc{False}.

Suppose first that $m = n$.
For the conjugacy test, return $c_Auc_B^{-1}$.
For the centraliser computation, let $C$ be the set of generators for
$C_G(A')$ found using \textsc{FindCentraliserGenerators},
and return $\{c_Awuc_B^{-1} : w \in C\}$.

So suppose that $m>n$.
Use \textsc{FindCentraliserExp} to find $C_G(A')$ as a finite set $C$ of
words of length $O(\mu)$. Note that $|C| \in O(1)$ by
Proposition~\ref{proposition:finite_small_ball}.
Check if $A^{c_Awu} = B^{c_B}$ for each word
$w \in C$.  Each check takes time $O(m\mu)$, so this part executes in time
$O(m\mu)$.
For the conjugacy test, return either the first element $c_Awuc_B^{-1}$ for
which this check succeeds, or \textsc{False} if no such element exists.
For the centraliser calculation, return the set of all elements
$c_Awuc_B^{-1}$ for which that this check succeeds.

Since each part of the algorithm takes time $O(m\mu)$,
Theorems 1 and 2 are proved.

\bibliography{bibitems}
\bibliographystyle{plain}

\end{document}